\documentclass[a4paper,11pt]{article}
\usepackage[utf8x]{inputenc}
\usepackage{amsmath, amsfonts,amssymb,mathtools,amsthm,mathrsfs,bbm}
\usepackage{bbding, pifont, dsfont}
\usepackage{graphicx}
\usepackage{xcolor,a4wide}
\usepackage[rightcaption]{sidecap}
\usepackage{csquotes}
\usepackage{tikz}
\usepackage{hhline}

\definecolor{darkgrey}{rgb}{0.75,0.75,0.75}

\newtheorem{theorem}{Theorem}
\newtheorem{proposition}{Proposition}[section]
\newtheorem{lemma}[proposition]{Lemma}

\newtheorem{definition}[proposition]{Definition}
\theoremstyle{definition}
\newtheorem{remark}[proposition]{Remark}

\newtheorem{assumption}[proposition]{Assumption}

\newcommand\unnumberedfootnote[1]{ %
        \let\temp=\thefootnote %
        \renewcommand{\thefootnote}{}%
        \footnote{#1}%
        \let\thefootnote=\temp%
        \addtocounter{footnote}{-1}}

\makeatletter
\renewcommand{\@fnsymbol}[1]{\ensuremath{%
   \ifcase#1\or 1\or 2\or 3\or
   \mathsection\or \mathparagraph\or \|\or 1\or
   2\or 3 \else\@ctrerr\fi}}
\makeatother

\numberwithin{equation}{section}

\newcommand{\cadlag}{c\`adl\`ag}
\newcommand{\id}{\mathds{1}}
\newcommand{\binary}{\{0,1\}}

\title{A spatial model for selection and cooperation}
\author{Peter Czuppon, Peter Pfaffelhuber} 
\date{\today}

\begin{document}

\maketitle

\begin{abstract}
  \noindent   We study the evolution of cooperation in an interacting
  particle system with two types. The model we investigate is an
  extension of a two-type biased voter model. One type (called
  defector) has a (positive) bias $\alpha$ with respect to the other
  type (called cooperator). However, a cooperator helps a neighbor
  (either defector or cooperator) to reproduce at rate $\gamma$.  We
  prove that the one-dimensional nearest-neighbor interacting
  dynamical system exhibits a phase transition at $\alpha=\gamma$. A special choice of interaction kernels yield that for
  $\alpha>\gamma$ cooperators always die out, but if $\gamma>\alpha$,
  cooperation is the winning strategy.
\end{abstract}



\section{Introduction}
In nature cooperative behavior amongst individuals is widely spread. It is observed in animals, e.g. \cite{clutton,griffin}, as well as in microorganisms, e.g.
\cite{crespi,levin}. In the attempt to understand this
phenomenon by models, theoretical approaches introduced different
interpretations and forms of cooperation, mostly within the area of
game theory \cite{nowak}. In all such approaches, a defector (or
selfish) type tends to have more offspring, but there are cases when
it is outcompeted by the cooperator type under some circumstances.
Although in all of the models describing cooperation the question of
extinction and survival of a type or the coexistence of several types
are main subjects of the mathematical analysis, the frameworks for the
theoretical studies may vary. While (stochastic) differential
equations are mainly used for non-spatial systems (see for example
\cite{archetti,hutzenthaler}), the theory of interacting particle
systems provides a suitable setup for the analysis of models with
local interactions between the particles,
\cite{blath,lanchier,sturm}. In this paper we define a model using the
latter structure and terminology.

Investigations of models incorporating cooperation are interesting
because of the following dichotomy: in non-spatial (well-mixed)
situations, the whole population benefits from the cooperative
behavior. If defectors have a higher fitness than cooperators,
defectors always outcompete cooperators in the long run. However, if
the system is truly spatial, cooperators can form clusters and then
use their cooperative behavior in order to defend themselves against
defectors, even though those might have higher (individual)
fitness. This heuristics suggests that only structured models can
help to understand cooperative behavior in nature.  For the model
studied in the present paper, we will make it precise in Proposition
\ref{prop:unstructured} for extinction of cooperators in a non-spatial
system and in Theorem~\ref{thm:ext_surv} for extinction of defectors
in a spatial system, if cooperation is strong enough.

Due to the variety of interpretations of cooperative behavior there
are different ways of implementing these mechanisms in a spatial
context. In the field of population dynamics, Sturm and Swart
\cite{sturm} study an interacting particle system containing a
cooperative-branching mechanism which can be understood as a sexual
reproduction event. In \cite{blath}, Blath and Kurt study a
branching-annihilating random walk and again, a cooperation mechanism
is interpreted as sexual reproduction. In contrast, the model
introduced by Evilsizor and Lanchier in \cite{lanchier} originates
from the game-theoretical study of a two player game with different
strategies where the strategies can be altruistic or selfish. Here,
the altruistic strategies represent the cooperator type. We discuss
the findings of these models to our results in
Section~\ref{sec:comparison}.

Various interacting particle systems which appear in the literature
are attractive, i.e.\ two versions of the system, which start in
configurations where one dominates the other, can be coupled such that
this property holds for all times; see e.g.\ \cite{sturm} for an
attractive model mentioned above. For such processes, there exist
several general results (cf. \cite{liggett}) which provide some useful
techniques helping in the analysis.  However, cooperation often leads
to non-attractive interacting particle systems; see
\cite{blath,lanchier} and the one presented here. The reason here is
that cooperators (or altruists) do not distinguish between
non-cooperators and their own type which usually contradicts
attractiveness.

The motivation for the present paper came from studies of bacterial
cells in the context of public-good-dilemmas, e.g. \cite{brockhurst,
  drescher}. The idea is that there are two types (defector=0,
cooperator=1), where only cooperators produce some public good which
helps neighboring cells to reproduce. However, this production is
costly which means that defectors will have a selective advantage over
the cooperator type. The resulting model is a biased voter model with
an additional cooperation mechanism. The main objective of our paper
is to study the long-time behavior of such a model dependent on the
parameters of the system. 

In particular, we prove for our main model in one dimension from
Definition~\ref{def:model}.3, that the system clusters independently
of the parameter configuration. When starting in a translation
invariant configuration, for $\alpha>\gamma$, defectors take over the
population, whereas for $\gamma>\alpha$ cooperators win; see
Theorem~\ref{thm:ext_surv}. Additionally, in higher dimensions, at
least we can show that the parameter region where defecting particles
win is larger than for $d=1$; see Theorem~\ref{thm:ext_surv2}. We also
show that a finite number of cooperators dies out if $\alpha>\gamma$,
but may survive if $\gamma>\alpha$. The converse holds true for 
defectors; see Theorem~\ref{thm:complete_conv}.  What remains to be
seen is if there are parameter combinations such that cooperators win
also in higher dimensions. Some preliminary results in the limit of
small parameters $\alpha$ and $\gamma$ or very large $\gamma$ can be found in
\cite{dissertation}.

The paper is structured as follows. First, we give a general
definition of the model in Section \ref{sec:model}. After the
definition we derive some properties of the model, show its existence
and consider some special cases and related systems. In Section
\ref{sec:result} we state limit results for the main model and its
derivatives, mainly restricted to the one-dimensional
lattice. Subsequently, in Section \ref{sec:comparison}, we compare our
results with those obtained in similar models, e.g.\ from \cite{blath}
and \cite{lanchier}. The rest of the paper is devoted to the proofs of
the theorems.

\section{The model and first results}\label{sec:model}
\subsection{The model}
Let $V$ be a countable vertex set, and $(a(u,v))_{u,v\in V}$ be a (not
necessarily symmetric) Markov kernel from $V$ to $V$. Additionally,
$(b(u,(v,w))_{u\in V, (v,w)\in V\times V}$ is a second Markov kernel
from $V$ to $V\times V$. We study an interacting particle system
$X=((X_t(u))_{u\in V})_{t\geq 0}$ with state space $\{0,1\}^V$, where
$X_t(u)\in\{0,1\}$ is the type at site $u$ at time $t$. A particle in
state~0 is called {\it defector} and a particle in state~1 is
called {\it cooperator}. The dynamics of the interacting particle
system, which is a Markov process, is (informally) as follows: For
some $\alpha,\gamma\geq 0$:
\begin{itemize}
\item \textit{Reproduction:} A particle at site $u\in V$ reproduces
  with rate $a(u,v)$ to site $v$, i.e.\ $X(v)$ changes to
  $X(u)$. (This mechanism is well-known from the voter model.)
\item \textit{Selection:} If $X(u)=0$ (i.e.\ there is a {\it defector}
  at site $u\in V$), it reproduces with additional rate $\alpha\
  a(u,v)$ to site $v$, i.e.\ $X(v)$ changes to $0$.  (A defector has a
  fitness advantage over the cooperators by this additional chance to
  reproduce. This mechanism is well-known from the biased
  voter model.)
\item \textit{Cooperation:} If $X(u)=1$ (i.e.\ there is a {\it
    cooperator} at site $u\in V$), the individual at site $v$ (no
  matter which state it has) reproduces to site $w$ at rate $\gamma\
  b(u,(v,w))\geq 0$.  (A cooperator at site $u$ helps an individual at
  site $v$ to reproduce to site $w$.)
\end{itemize}

\begin{remark}[Interpretation]
  \begin{enumerate}
  \item {\it Selection:} Since cooperation imposes an energetic cost
    on cooperators, the non-cooperating individuals can use these
    free resources for reproduction processes. This leads to a fitness
    advantage which we describe with the parameter $\alpha$.
  \item {\it Cooperation:} The idea of the cooperation mechanism in
    our model is that each cooperator supports a neighboring
    individual, independent of its type, to reproduce to another
    site according to the Markov kernel $b$. A biological
    interpretation for this supportive interaction is a common good
    produced by cooperators and released to the environment helping
    the colony to expand. The corresponding interaction parameter
    is~$\gamma$.\\
    Below, we will deal with two situations, depending on whether $b(u,
    (v,u))>0$ or $b(u,(v,u))=0$. In the former case, we speak of an
    altruistic system, since a cooperator at site $u$ can help the
    particle at site $v$ to kill it. In the latter case, we speak
    of a cooperative system.
  \end{enumerate}
\end{remark}

\noindent
In order to uniquely define a Markov process, we will need the
following assumption.

\begin{assumption}[Markov kernels\label{ass:kernel}]
  The Markov kernels $a(.,.)$ and $b(.,(.,.))$ satisfy
  \begin{align}\label{eq:ex_cond_1}
    \sum_{u\in V} a(u,v) <\infty \text{ for all } v\in V
  \intertext{and}
  \label{eq:ex_cond_2}
    \sum_{u,v\in V} b(u,(v,w)) <\infty \text{ for all } w\in V.
  \end{align}
\end{assumption}

\begin{remark}[Some special cases]
  A special case is
  \begin{align}
    \label{eq:100}
    b(u,(v,w)) = a(u,v)\cdot a(v,w) \quad \text{for all } u,v,w\in V.
  \end{align}
  Then, \eqref{eq:ex_cond_2} is implied by the assumption
  \begin{align*}
    \sup_{v\in V}\sum_{u\in V} a(u,v) <\infty,
  \end{align*}
  which is stronger than \eqref{eq:ex_cond_1}.  We will also deal with
  a similar case setting $b(u,(v,u))=0$ which means that $u$ cannot
  help $v$ to replace $u$. To be more precise, we set
  \begin{equation}
    \label{eq:102}
    b(u,(v,w)) = \left\{ \begin{array}{ll}  a(u,v)\cdot \frac{a(v,w) \id_{\{w\neq u\}}}{\sum_{w'\neq u}a(v,w')}, 
                           & \text{if } a(v,u)<1, \\ 0, & \text{else,}\end{array}\right. \quad\text{for all } u,v,w\in V.
   \end{equation}
   The normalizing sum in the denominator emerges from the exclusion
   of self-replacement, i.e.\ \eqref{eq:102} is the two-step
   transition kernel of a self-avoiding random walk.
\end{remark}

\subsection{Existence and uniqueness of the process}
\noindent
We now become more formal and define the (pre-)generator of the
process $X$ via its transition rates. Given $X\in\binary^V$, the rate
of change $c(u,X)$ from $X$ to
\[ X^u(v) = \begin{cases} X(v), & v\in V\backslash \{u\}; \\
  1-X(u), & v=u;\end{cases} \]
is as follows: \\
If $X(u)=0$, then
\begin{align} 
  \label{eq:104}
  c(u,X)& =\sum_{v} a(v,u) X(v)+\gamma \sum_{v}
  X(v) \sum_{w} X(w) b(w,(v,u)).  \intertext{If
    $X(u)=1$, then} \label{eq:105} c(u,X) & = \left(1+\alpha\right)\sum_{v}
  a(v,u)(1-X(v))+\gamma \sum_{v}
  (1-X(v)) \sum_{w}
  X(w) b(w,(v,u)).
\end{align}
Here, the first sum in $c(u,X)$ represents the rates triggered by
reproduction and selection whereas the last terms emerge from the
cooperation mechanism.

The existence of a unique Markov process corresponding to the
transition rates $c(u,X)$ satisfying Assumption~\ref{ass:kernel} is guaranteed by standard theory, see for
example \cite[Chapter 1]{liggett}. Precisely, we define the
\textit{(pre-)generator} $\Omega$ of the process through
\begin{equation*}\label{eq:generator}
  (\Omega f)(X)=\sum_{u\in G} c(u,X)(f(X^u)-f(X)),
\end{equation*}
where $f\in \mathcal{D}(\Omega)$, the domain of $\Omega$, is given by
\begin{equation*}
  \mathcal{D}(\Omega):=\{f:\binary^V \rightarrow \mathbb{R} \text{ depends only on finitely many coordinates}\}.
\end{equation*}
We note that $\mathcal{D}(\Omega)$ is dense in $C_b(\binary^V)$, the set of
bounded continuous functions on $\binary^V$, because of the
Stone-Weierstrass-Theorem. We find the following general statement.

\begin{proposition}[Existence of unique Markov
  process\label{prop:existence}]
  If Assumption~\ref{ass:kernel} holds, the transition rates $c(.,.)$
  given in \eqref{eq:104} and \eqref{eq:105} define a unique Markov
  process $X$ on $\binary^V$. Moreover, the closure $\bar\Omega$ of
  $\Omega$ is the generator of $X$.
\end{proposition}

\begin{proof}
  We need to show that the closure of $\Omega$ in $C(\binary^V)$ is a
  generator of a semi-group which then uniquely defines a Markov
  process (see for example \cite[Theorem 1.1.5]{liggett}). In order to
  show this we follow \cite[Theorem 1.3.9]{liggett} and check the
  following two conditions:
  \begin{align}
    \label{eq:200}
    \sup_{u\in V} \sup_{X\in\binary^V} c(u,X) & < \infty, \\
    \label{eq:201}
    \sup_{u\in V} \sum_{v\neq u} \widetilde{c}_u(v)& < \infty,
  \end{align}
  where 
  $$\widetilde{c}_u(v):= \sup\{ \|
  c(u,X_1)-c(u,X_2)\|_T : X_1(w)=X_2(w) \text{ for all } w\neq v\}$$
  measures the dependence of the transition rate $c(u,X)$ of the site
  $v\in V$ and $\|\cdot\|_T$ denotes the total variation norm.
  \par
  Both inequalities follow from Assumption~\ref{ass:kernel} and the
  definition of the transition rates $c(.,.)$. Using these we obtain
  for any $X\in\binary^V$ and $u\in V$
  \begin{align*}
    c(u,X)\leq (1+\alpha) \sum_{v\in V} a(v,u) + \gamma \sum_{v,w\in
      V} b(w,(v,u)) < \infty
  \end{align*}
  showing \eqref{eq:200}.  For \eqref{eq:201}, we note that
  $\widetilde{c}_u(v)\neq 0$ only when either $a(v,u)>0$ or $b(w,(v,u))>0$
  or $b(v,(w,u))>0$ for some $w\in V$. Hence, for all $u\in V$ we
  obtain
  \begin{align*}
    \sum_{v\neq u} \widetilde{c}_u(v) & \leq \sum_{v\neq u}\left(
      (1+\alpha)a(v,u)+\gamma \sum_{w\in V} b(w,(v,u))+b(v,(w,u))
    \right)\\
    & \leq \sum_{v\in V} (1+\alpha)a(v,u) + 2\gamma \sum_{v,w\in V}
    b(v,(w,u)) <\infty,
  \end{align*}
  where we used the inequalities \eqref{eq:ex_cond_1} and
  \eqref{eq:ex_cond_2} again and we have proved~\eqref{eq:201}.

  Now, using \cite[Theorem 1.3.9]{liggett} we see that the closure of
  $\Omega$ in $C(\binary^V)$ is a Markov generator of a Markov
  semigroup. This finishes the proof.
\end{proof}

\noindent
We can now define the voter model with bias and cooperation.

\begin{definition}[(Cooperative/Altruistic) Voter Model with Bias and
  Cooperation]\label{def:model}
  Let $a(.,.)$ be a Markov kernel from $V$ to $V$ satisfying
  \eqref{eq:ex_cond_1} and $b(.,(.,.))$ be a Markov kernel from $V$ to
  $V\times V$ satisfying \eqref{eq:ex_cond_2}.
  \begin{enumerate}
  \item The (unique) Markov process with transition rates given by
    \eqref{eq:104} and \eqref{eq:105} is called the \emph{Voter Model
      with Bias and Cooperation} (VMBC).
  \item If \eqref{eq:100}  holds, the VMBC is called the
    \emph{altruistic Voter Model with Bias and Cooperation} (aVMBC).  
  \item If \eqref{eq:102} holds, the VMBC is called the
    \emph{cooperative Voter Model with Bias and Cooperation} (cVMBC).
  \end{enumerate}
\end{definition}

\subsection{Unstructured populations}
\noindent
As a first result, we show that the probability for cooperators to die
out on a large, complete graph tends to one (for $\alpha>0$). We
consider the special case of an {\it unstructured
  population}. Therefore, let $V^N$ be the vertex set of a graph with
$|V^N|=N$ and
$$ a^N(u,v) = \frac{1}{N-1} $$
for $u,v\in V^N$ with $u\neq v$. Due to the global neighborhood it is
equally likely to find configurations of the form "101" and
"110". Hence, cooperation events favoring a defector or a cooperator
happen with the same rate and thus cancel out when looking at the mean
field behavior of the system. We will show that defectors always take
over the system for large $N$. It can easily be seen that the aVMBC is
dominated by the cVMBC, so it suffices to show extinction of
cooperators for the cVMBC, i.e.\ we have
$$ b^N(u,(v,w)) = \frac{\id_{\{u\neq v\}}}{N-1} \frac{\id_{\{v\neq
    w\}}\id_{\{w\neq u\}}}{(N-1) \frac{N-2}{N-1}} =
\frac{1}{(N-1)(N-2)} \id_{\{u,v,w \text{ different\}}}.$$
We prove that in the limit for large $N$ the frequency of cooperators
follows a logistic equation with negative drift, hence cooperators die
out. See also \cite[Chapter 11]{ethier}.

\begin{proposition}[Convergence in the unstructured case]\label{prop:unstructured}
  Let $X^N$ be a cVMBC on $V^N$ and $S^N := \tfrac 1N \sum_u X^N(u)$ the
  frequency of cooperators. Then, if $S^N_0 \xRightarrow{N\to\infty} s_0$,
  then
  $$ S^N \xRightarrow{N\to\infty} S,$$
  where $S$ solves the ODE
  \begin{equation*}
  	 dS = -\alpha S(1-S)
  \end{equation*}
  with $S_0=s_0$, independently of $\gamma$.
\end{proposition}

\begin{proof}
  In order to prove the limiting behavior for $N\rightarrow\infty$, we
  observe that $S^N$ is a Markov process. A calculation of the
  generator $\Omega^N$ applied to some smooth function~$f$ yields
  \begin{align*}
    \Omega^N f(s) & = N s \frac{1-s}{1-\tfrac 1N}(f(s + \tfrac 1N) -
    f(s)) + (1+\alpha) N(1-s) \frac{s}{1-\tfrac 1N}(f(s - \tfrac 1N) -
    f(s)) \\ & \qquad \qquad + \gamma Ns \frac{s - \tfrac 1N}{1-\tfrac
      1N}\frac{1-s}{1-\tfrac 2N}(f(s + \tfrac 1N) - f(s)) \\ & \qquad
    \qquad \qquad \qquad + \gamma Ns\frac{1-s}{1-\tfrac
      1N}\frac{s-\tfrac 1N}{1-\tfrac 2N}(f(s - \tfrac 1N) - f(s)) \\ &
    \xrightarrow{N\to\infty} -\alpha s(1-s) f'(s).
  \end{align*}
  Applying standard weak convergence results, see for example \cite[Theorem 4.8.2]{ethier}, this shows the claimed convergence.
\end{proof}

\section{Results on the long-time behavior for $V=\mathbb Z^d$}
\label{sec:result}
Our main goal is to derive the long-time behavior of the VMBC with
$V=\mathbb Z^d$. In spin-flip systems, results on the ergodic
behavior can be obtained by general principles if the process is {\it
  attractive}. Thereby, a spin-system is called attractive if for two
configurations $X,Y\in \binary^V$ with $X\leq Y$ componentwise, the
corresponding transition rates $c$ satisfy the following two relations
for all $u\in V$
\begin{align}\label{eq:monotonicity1}
  X(u)=Y(u)=0 & \Rightarrow c(u,X)\leq c(u,Y),\\
  X(u)=Y(u)=1 & \Rightarrow c(u,X)\geq
  c(u,Y).\label{eq:monotonicity2}
\end{align}
However, the VMBC is not attractive for $\gamma>0$. Indeed, consider the
simple case when $V=\{u,v,w\}$ with Markov kernels
$$ a(u,v) = a(v,w) = a(w,u)=1$$ and $b(u, (v,w)) = a(u,v)a(v,w)$. 
Then, let $X=(001)$ and $Y=(101)$ (i.e.\ $X(u)= 0, Y(u)=1,
X(v)=Y(v)=0, X(w)=Y(w)=1$) and note that $X\leq Y$, but
\begin{align*}
  c(w, X) & = 1+\alpha < 1+\alpha + \gamma = c(w,Y).
\end{align*}
This shows that \eqref{eq:monotonicity2} is not satisfied at $w\in
V$.
Hence, proofs for the long-time behavior require other strategies
which do not rely on attractiveness of the process.

\noindent
Before we state our main results we define what we mean by extinction
and clustering.

\begin{definition}[Extinction, clustering]
  \begin{enumerate}
  \item We say that in the VMBC-process $(X_t)_{t\geq 0}$ type $i
    \in\binary$ \emph{dies out} if
    \begin{align*}
      P\left(\lim_{t\rightarrow\infty} X_t(u) = 1-i\right)=1, \quad \text{for all } u\in V.
    \end{align*} 
  \item We say that the VMBC-process \emph{clusters} if for all
    $u,v\in V$
    \begin{align*}
      \lim_{t\rightarrow\infty} P(X_t(u) = X_t(v))=1.   	 
    \end{align*} 
  \end{enumerate}
\end{definition}

\noindent
We will use $V=\mathbb Z^d$ and nearest neighbor interaction via the
kernels $a$ and $b$. In this case we have that for all
$u,v,w\in \mathbb Z^d$ with $|u-v| = |w-v|=1$
\begin{align}
  a(u,v) & = \frac{1}{2d}, \qquad  b(u,(v,w)) = \frac{1}{(2d)^2} \label{eq:avmbc}
           \intertext{for the aVMBC and}
           a(u,v) & = \frac{1}{2d}, \qquad  b(u,(v,w)) = \frac{1}{2d(2d-1)} \id_{\{u\neq w\}}
\end{align}
for the cVMBC. We say that (the distribution of) a
$\binary^{\mathbb{Z}^d}$-valued random configuration $X$ is
\emph{non-trivial} if $P(X(u)=0\text{ for all $u$}), P(X(u)=1\text{
  for all $u$}) <1$.
Furthermore, we call $X$ \emph{translation invariant} if
$(X(u_1),...,X(u_n)) \stackrel d =(X(u_1+v),...,X(u_n+v))$ for all
$n\in\mathbb{N}, u_1,...,u_n,v\in \mathbb Z^d$. If the VMBC model is
started in a translation invariant configuration
$X_0\in\binary^{\mathbb Z^d}$, the configuration $X_t$ is translation
invariant due to the homogeneous model dynamics.

Now we can state our main results. 
For cVMBC, we
distinguish between the case $\alpha>\gamma$ where we can state a
convergence result in all dimensions $d\geq 1$, the case
$\gamma>\alpha$ and the case $\gamma=\alpha$. In the last two cases,
the method of proof is only applicable in dimension $d=1$.

\begin{theorem}[cVMBC-limits]\label{thm:ext_surv}
  Let $V=\mathbb Z^d$ and $a(.,.)$ be the nearest neighbor random walk
  kernel and $X$ be the cVMBC with $\alpha,\gamma \geq 0$
  starting in some non-trivial translation invariant configuration.
  \begin{enumerate}
  \item[(i)] If $d\geq 1$ and $\alpha>\gamma$, the cooperators die
    out.
  \item[(ii)] If $d=1$ and $\gamma>\alpha$, the defectors die out.
  \item[(iii)] If $d=1$ and $\gamma=\alpha$, the process clusters.
  \end{enumerate}
\end{theorem} 

\noindent
The proof of Theorem~\ref{thm:ext_surv} can be found in
Section~\ref{sec:proof1}. Briefly, for $\alpha>\gamma$, we will use a
comparison argument with a biased voter model, see Definition
\ref{def:biasedvoter}. For $\gamma>\alpha$ and $d=1$, however, we
prove the convergence result with the help of a clustersize-process
which takes the special form of a one-dimensional jump process. As we
will see, for $\gamma>\alpha$, a cluster of cooperators has a positive
probability to survive and expand to infinity which will then yield
the result. Unfortunately, due to the simple description of such a
cluster in one dimension, this argument cannot be extended to higher
dimensions. However, resorting to some simulation results for $d=2$
and $d=3$, we see a similar behavior (with a different threshold) like
in $d=1$, see Figure~\ref{fig:simulation}. For higher dimensions,
spatial correlations between sites are weaker reducing the impact of
clusters on the evolution of the system. This in turn leads to a
reduced chance of survival of cooperators.

\begin{figure}
  \begin{center}
    \includegraphics[width=0.4\textwidth]{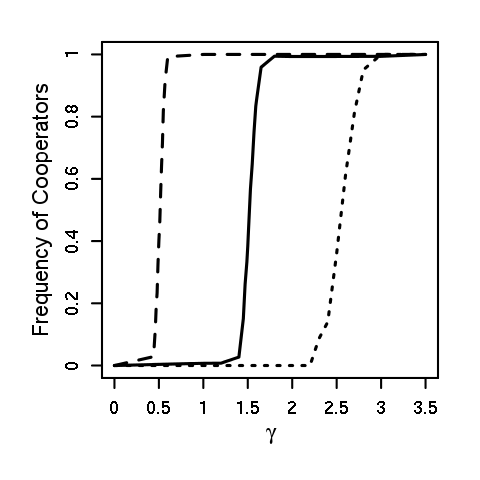}
    \caption{Relative frequencies of cooperators after $100,000$
      transitions of the cVMBC on a $1,000$ sites torus in one
      dimension (dashed line), a $40\times 40$ sites torus in two
      dimensions (solid line) and a $12\times 12\times 12$ sites torus
      in three dimensions (dotted line). The initial configuration was
      a Bernoulli-product measure with probability $0.5$ and the
      selection rate $\alpha$ was set to $0.5$. We suspect that the slightly
      smaller slope in three dimensions is a finite-number effect.}
    \label{fig:simulation}
  \end{center}			
\end{figure} 

\noindent
For the aVMBC, we can only state a threshold when cooperators die out.

\begin{theorem}[aVMBC-limits]\label{thm:ext_surv2}
  Let $V=\mathbb Z^d$ and $a(.,.)$ be the nearest neighbor random walk
  kernel and $X$ be the aVMBC with $\alpha,\gamma \geq 0$ starting in
  some non-trivial translation invariant configuration. 
  \begin{enumerate}
  \item[(i)] If $d\geq 1$ and $\alpha>\gamma\frac{d-1}{d},$ the cooperators
    die out. In particular, for $d=1$, the cooperators die out if
    $\alpha>0$ independently of $\gamma$.
  \item[(ii)] If $d=1$, the process equals the cVMBC with parameters $\alpha
    + \gamma/2$ and $\gamma/2$ in distribution. In particular, if
    $\gamma>\alpha=0$, the process clusters.
  \end{enumerate}
\end{theorem}

The proof of the Theorem can be found in
Section~\ref{sec:proof2}. Again, for $\alpha>\gamma(d-1)/d$, we can
use a comparison argument with the biased voter model. However, it
remains an open question whether cooperators in the aVMBC have a
positive probability of survival in any dimension. On the one hand,
the difference between the aVMBC and the cVMBC becomes smaller in high
dimensions suggesting survival of cooperators for large $\gamma$. On the other hand, clustering is usually more difficult in
higher dimensions but cooperators can only survive due to
clustering. First simulation results for $d=2$ and $d=3$ show that
survival of cooperators is unlikely in the aVMBC.

\begin{remark}[Cooperation only among cooperators]
  Another cooperation mechanism we might consider arises if
  cooperators only help other cooperators, i.e. the cells recognize
  related cells. In ecological literature this behavior is called
  kin-recognition or kin-discrimination, see \cite{penn} for an overview. As to the theoretical
  behavior of the model this changes the transition rate in
  \eqref{eq:105}, i.e. if $X(u)=1$ then
  \begin{equation*}
    \begin{aligned}
      c(u,X) & = \left(1+\alpha\right) \sum_{v} a(v,u)(1-X(v)).
    \end{aligned}     
  \end{equation*}
  Here, cooperators are less likely to die and hence, this process
  dominates the cVMBC. In particular, for translation invariant
  initial conditions, defectors die out for $\gamma>\alpha$ in one
  dimension. Moreover, as can be seen from a calculation similar as in
  the proof of Lemma~\ref{l:dominance}, a biased voter model, where
  type~0 is favored, still dominates this process for
  $\alpha>\gamma$. Hence, we also have that cooperators die out in
  this case and the same results as in Theorem~\ref{thm:ext_surv}
  hold.
\end{remark}

Since cooperators always die out in $d=1$ for the aVMBC (as long as
$\alpha>0$), we focus on the cVMBC in the sequel. We state some
results if the starting configuration is not translation invariant,
but contains only a finite number of cooperators or defectors.

\begin{theorem}[Finite initial
  configurations]\label{thm:complete_conv}
  Let $V=\mathbb{Z}$ and $a(.,.)$ be the nearest neighbor random walk
  kernel and $X$ be the cVMBC with $\alpha,\gamma\geq 0$. Let $X_0$
  contain either finitely many defectors or finitely many cooperators
  (i.e.\ $X_0 = \id_A$ or $X_0=1-\id_A$ for some finite $A\subseteq V$).
  \begin{enumerate}
  \item[(i)] The process clusters. 
  \item[(ii)] If $\alpha\geq\gamma$ and $X_0$ contains finitely many
    cooperators, the cooperators die out.
  \item[(iii)] If $\gamma\geq\alpha$ and $X_0$ contains finitely many
    defectors, the defectors die out.
  \end{enumerate}
\end{theorem}

\begin{remark}[Starting with a single particle]
  A particularly simple initial condition is given if $|A|=1$. In case
  there is initially only a single cooperator, we note that the size of the
  cluster of cooperators $(C_t)_{t\geq 0}$ is a birth-death process
  which jumps from $C$ to
  \begin{align*}
    & C+1 \text{ at rate } \id_{\{C>0\}} + \gamma \cdot \id_{\{C\geq 2\}}, \\ &
    C-1 \text{ at rate } (1+\alpha)\cdot \id_{\{C>0\}}.
  \end{align*}
  Conversely, if there is only a single defector, the size of the
  cluster of defectors $(D_t)_{t\geq 0}$ is a birth-death process
  which jumps from $D$ to
  \begin{align*}
    & D+1 \text{ at rate } (1+\alpha) \cdot \id_{\{D>0\}} + \gamma\cdot \id_{\{D=1\}}, \\ & D-1 \text{
      at rate } (1+\gamma)\cdot \id_{\{D>0\}}.
  \end{align*}
  Hence, either cooperators or defectors die out, depending on whether
  $(C_t)_{t\geq 0}$ (or $(D_t)_{t\geq 0}$) hits~0 or not.
\end{remark}

The proof of Theorem~\ref{thm:complete_conv} is given in Section
\ref{sec:proof3}. Note that the only situations where the process
does not converge to a deterministic configuration in this setting are
the cases where $\gamma>\alpha$ ($\alpha>\gamma$) and the process
starts with finitely many cooperators (defectors). Here, the limit
distribution is a linear combination of the invariant measures
$\delta_{\underline{0}}$ and $\delta_{\underline{1}}$, the all-zero and all-one configuration, respectively. This basically
means that we observe clustering, which is statement (i) above.

\section{Comparison to results from  \cite{blath} and \cite{lanchier}}
\label{sec:comparison}
In this section we compare our results on the cVMBC to those obtained by Blath and Kurt in \cite{blath} and the system introduced by Evilsizor and Lanchier in \cite{lanchier}. We choose these two models since both have mechanisms favoring one type, while  a second type is only favored if it occurs in a cluster.

\subsection*{Comparison to \cite{blath}}
One model studied is the cooperative caring double-branching annihilating random walk (ccDBARW) on the integer lattice $\mathbb{Z}$. Particles migrate to neighboring sites with rate $m$ and annihilate when meeting another particle. (Note that this mechanism favors the unoccupied state.) The double-branching events happen with rate $1-m$. Here, the authors restrict branching to particles with an occupied neighboring site and such particles branch to the next unoccupied site to the left and to the right. (That is, if a cluster of size $\geq 2$ already exists, the branching mechanism extends the cluster.) Their result about this process, starting in a finite configuration (see Theorem 2.4 in the paper), states that for $m<1/2$, particles survive with positive probability, whereas for $m>2/3$, particles die out almost surely. 

Although Blath and Kurt only discuss the case of a finite initial configuration, the results are in line with our findings: If the mechanism to favor enlargement of existing clusters (cooperation in our case and cooperative branching in their case) is too weak, type~0 (or the unoccupied state) wins. Importantly, in both models, enlargement of existing clusters can be strong enough in order to outcompete the beneficial (or unoccupied) type.

\subsection*{Comparison to \cite{lanchier}}
The model studied in \cite{lanchier} -- called the death-birth updating process -- emerges from a game theoretic model with two strategies. This means that transition rates are derived from a $2\times 2$ payoff-matrix with entries $a_{ij}$ for $i,j\in\{1,2\}$ representing the payoff obtained by a particle of type $i$ due to interacting with a particle of type $j$. Now, a particle dies with rate $1$ and is replaced by a particle in its neighborhood proportionally to its fitness which is determined by the values of the payoff-matrix. The neighborhood is given by blocks of radius $R$. The authors analyze this model in different settings. They call a strategy $i$ selfish if $a_{ii} > a_{ji}$ for $j\neq i$ (i.e.\ the payoff having strategy~$i$ as opponent is bigger if one has the same strategy~$i$) and altruistic if $a_{ii} < a_{ji}$. Again, in a non-spatial version of this game, selfish strategies always outcompete altruistic strategies. Noting that selfish strategies seem to be fitter, altruistic strategies might become favorable if they form a big cluster because altruists might have a high payoff. As the results in~\cite{lanchier} show (see their Figure~2), there are parameter regions -- in particular in a spatial Prisoner's dilemma -- where altruists can outcompete selfish strategies. 

Clearly, the cVMBC is a much simpler model than the death-birth updating process. This is resembled in the results, since \cite{lanchier} show parameter combinations with coexistence for the death-birth updating process, but our results never show coexistence for the cVMBC. However, as in our findings for the cVMBC, \cite{lanchier} find that types unfavorable in a non-spatial context can indeed win in all dimensions. Unfortunately, they can only give bounds on the phase transition in their model, while we have seen that $\alpha = \gamma$ is a sharp threshold, at least in one dimension.

\section{Preliminaries}
Here we provide some useful results for the proofs of our theorems. In
particular, we provide a comparison with a biased voter model in
Section~\ref{sec:comp} and a particular jump process in
Section~\ref{sec:jump}. 

\subsection{Comparison results}
\label{sec:comp}
In cases where $\alpha>\gamma$, it is possible to prove a stochastic
domination of the VMBC by a biased voter model. The precise statements
will be given below. But first, we define this process, which was
introduced by Williams and Bjerknes in \cite{williams} and first studied by Bramson and Griffeath in \cite{bramson}.

\begin{definition}[Biased Voter Model\label{def:biasedvoter}]
  The biased voter model with bias $\beta\geq -1$ and $\delta\geq -1$
  is a spin system $\widetilde X$ with state space
  $\{0,1\}^V$ and transition rates as follows:\\
  If $\widetilde X(u)=0$, then
  \begin{align*} 
    \widetilde c(u,\widetilde X) & =(1+\beta)\sum_{v}
    a(v,u) X(v). 
    \intertext{If $\widetilde X(u)=1$, then} 
	 \widetilde c(u,\widetilde X) & = (1+\delta)\sum_{v}
    a(v,u)(1-X(v))
  \end{align*}  
\end{definition}

\begin{remark}[Long-time behavior of the biased voter
  model\label{rem:biasvoter}]
  The long-time behavior of the biased voter model is quite simple.
  In \cite{bramson}, the limit behavior of the biased voter model in
  $V=\mathbb Z^d$ with nearest neighbor interactions is studied.
  Generalizations to the case of $d-$regular trees for $d\geq 3$
  can be found in \cite{louidor}. We restate the results for $V=\mathbb Z^d$:\\
  Let $\widetilde{X}$ be a biased voter model with bias $\beta>-1$ and
  $\delta>-1$ as introduced in Definition~\ref{def:biasedvoter}. For
  any configuration $X_0\in\binary^{\mathbb{Z}^d}$ with infinitely
  particles of each type it holds that the type with less bias dies out, i.e.:
  \begin{enumerate}
  \item If $\beta > \delta$, type 0 dies out (i.e.\
    $P(\lim_{t\rightarrow\infty} \widetilde{X}_t(u) = 1)=1$ for all $u\in V$).
  \item If $\delta > \beta$, type 1 dies out (i.e.\
    $P(\lim_{t\rightarrow\infty} \widetilde{X}_t(u) = 0)=1$ for all $u\in V$).
  \end{enumerate}
\end{remark}

\begin{lemma}[cVMBC$\leq$biased voter model\label{l:dominance}]
  Let $X$ be a cVMBC with bias $\alpha$ and cooperation coefficient
  $\gamma$ and $\widetilde X$ a biased voter model with bias $\gamma$ and
  $\alpha$. Then, if $b(.,(.,.))$ satisfies $\sum_u b(u,(v,w))\leq
  a(v,w)$, and $X_0\leq \widetilde X_0$, it is possible to couple $X$ and
  $\widetilde X$ such that $X_t\leq \widetilde X_t$ for all $t\geq 0$.
\end{lemma}

\begin{proof}
  We need to show (see \cite[Theorem 3.1.5]{liggett}) that for $X\leq \widetilde X$
  \begin{equation}
    \label{eq:303}
    \begin{aligned} 
      & \text{if } X(u)=\widetilde{X}(u)=0, &\text{then } c(u,X)\leq \widetilde{c}(u,\widetilde{X}),\\
      & \text{if } X(u)=\widetilde{X}(u)=1, &\text{then } c(x,X)\geq
      \widetilde{c}(u,\widetilde{X}).
    \end{aligned} 
  \end{equation}
  We start with the first assertion and write
  \begin{align*}
    c(u,X)& =\sum_{v} a(v,u) X(v)+\gamma \sum_{v}
    X(v) \sum_{w} X(w) b(w,(v,u)) \\ & \leq
    \sum_{v} a(v,u) X(v) + \gamma \sum_{v} X(v)
    a(v,u) \\ & \leq (1+\gamma)\sum_{v} a(v,u) \widetilde{X} (v)
    = \widetilde c(u,\widetilde X),\\
    \intertext{and for the second inequality we have} c(u,X) & =
    \left(1+\alpha\right) \sum_{v} a(v,u) (1-X(v)) +\gamma
    \sum_{v} (1-X(v)) \sum_{w,v} X(w) b(w,(v,u)) \\
    & \geq (1+\alpha)\sum_{v} a(v,u) (1-X(v)) \geq
    (1+\alpha)\sum_{v} a(v,u) (1-\widetilde{X}(v)) = \widetilde
    c(u,\widetilde X).
  \end{align*}
  This finishes the proof.
\end{proof}

\noindent
Next, we focus on the aVMBC in the case $V=\mathbb Z^d$ and the
symmetric, nearest-neighbor random walk kernel.

\begin{lemma}[aVMBC$\leq$biased voter model\label{l:44}]
  Let $V=\mathbb Z^d$, $a(.,.)$ be the nearest-neighbor random walk
  kernel defined in equation \eqref{eq:avmbc}, $X$ be an aVMBC with bias $\alpha$ and cooperation
  coefficient $\gamma$ and $\widetilde X$ a biased voter model with bias
  $\gamma(2d-1)/(2d)$ and $\alpha + \gamma/(2d)$. Then, if $X_0\leq
  \widetilde X_0$, it is possible to couple $X$ and $\widetilde X$ such that
  $X_t\leq \widetilde X_t$ for all $t\geq 0$.
\end{lemma}

\begin{proof}
  Again, we need to show that for $X\leq \widetilde X$ the inequalities in~\eqref{eq:303}
  hold. We start with the first assertion and write by using that $X(u)=0$
  \begin{align*}
    c(u,X)& =\sum_{v} a(v,u) X(v) +\gamma \sum_{v}
    X(v) \sum_{w} X(w) a(v,w)a(v,u) \\ & \leq
    \sum_{v} a(v,u) X(v) + \gamma \sum_{v} X(v)
    a(v,u) \sum_{w\neq u} a(v,w)\\ & \leq
    \Big(1+\gamma\frac{2d-1}{2d}\Big)\sum_{v} a(v,u) \widetilde{X} (v)
    = \widetilde c(u,\widetilde X),\\
    \intertext{and for the second inequality, now using $X(u)=1$ we have} c(u,X) & =
    \left(1+\alpha\right) \sum_{v} a(v,u) (1-X(v)) + \gamma
    \sum_{v} (1-X(v)) \sum_{w} X(w) a(w,v)a(v,u) \\
    & \geq (1+\alpha) \sum_{v} a(v,u) (1-X(v)) +\gamma
    \sum_{v} (1-X(v)) a(u,v)a(v,u) \\
    & = \Big(1+\alpha + \frac{\gamma}{2d}\Big)\sum_{v}
    a(v,u) (1-X(v)) \\ &  \geq \Big(1+\alpha +
    \frac{\gamma}{2d}\Big)\sum_{v} a(v,u) (1-\widetilde{X}(v)) =
    \widetilde c(u,\widetilde X).
  \end{align*}
  This yields the statement.
\end{proof}

\subsection{A result on a jump process}
\label{sec:jump}
In the proof of Theorem~\ref{thm:ext_surv}, we will use the dynamics
of the size of a cluster of cooperators and rely on a comparison of
this clustersize process with a certain jump process (which jumps
downward by at most one and upwards by at most two). The following
Proposition will be needed.

\begin{proposition}[A jump process\label{P:supermartingal}]
  Let
  $(\mu(t))_{t\geq 0}, (\lambda_1(t))_{t\geq 0}, (\lambda_2(t))_{t\geq
    0}$
  be $\mathbb R_+$-valued \cadlag -stochastic processes, adapted to
  some filtration $(\mathcal F_t)_{t\geq 0}$, which satisfy
  \begin{equation}
    \label{eq:812}
    \begin{aligned}
      \lambda_1(t) + 2\lambda_2(t)- \mu(t) & > \varepsilon>0 \text{ for some } \varepsilon \text{ and}\\
      \lambda_1(t) + \lambda_2(t) + \mu(t) & < C \text{ for some } C >0.
    \end{aligned}
  \end{equation}
  In addition, let $(C_t)_{t\geq 0}$ be a $\mathbb Z$-valued
  $(\mathcal F_t)_{t\geq 0}$-Markov-jump-process, which
  jumps at time $t$ from $x$ to
  \begin{align*}
    x-1 & \text{ at rate $\mu(t)$},\\
    x+1 & \text{ at rate $\lambda_1(t)$},\\
    x+2 & \text{ at rate $\lambda_2(t)$}
  \end{align*}
  Then, 
  \begin{enumerate}
  \item $C_t \xrightarrow{t\to\infty} \infty$ almost surely and
  \item $P(T_1=\infty)>0$, for $C_0=2$ and $T_1:=\inf\{t: C_t=1\}$.
  \end{enumerate}
\end{proposition}

\begin{proof}
  In the case of time-homogeneous rates, i.e.\ constant
  $\mu, \lambda_1$ and $\lambda_2$, the assertion is an immediate
  consequence of the law of large numbers. We prove the general case
  by using martingale theory. We assume without loss of generality
  that $\lambda_1(t)+\lambda_2(t)+\mu(t)=1$ for all $t\geq
  0$.
  (Otherwise, we use a time-rescaling. Note that this rescaling is
  bounded by assumption~\eqref{eq:812} and therefore,
  $C_t \xrightarrow{t\to\infty} \infty$ holds iff it holds for the
  rescaled process.)

  We first show that there exists $a_c>0$ such that for all
  $a\in (0,a_c)$, the process $(\exp(-a C_t))_{t\geq 0}$ is a positive
  $(\mathcal F_t)_{t\geq 0}$-super-martingale. For this, consider the
  (time-dependent) generator of the process $(C_t)_{t\geq 0}$ applied
  to the function $f(x)=\exp(-ax)$ which yields at time $t$
  \begin{equation*}
    \begin{aligned}
      (G_t^{\mathcal{C}}f)(x)&=\lambda_1(t)\exp(-a(x+1))+\lambda_2(t)\exp(-a(x+2))
      \\ & \qquad \qquad \qquad \qquad \qquad \qquad +\mu(t)\exp(-a(x-1))-\exp(-ax)\\
      &=\exp(-ax)(\lambda_1(t)\exp(-a)+\lambda_2(t)\exp(-2a)+\mu(t)\exp(a)-1)
      \\ & = \exp(-ax) g_t(a)
    \end{aligned}
  \end{equation*}
  for
  $g_t(a):=\lambda_1(t)\exp(-a)+\lambda_2(t)\exp(-2a)+\mu(t)\exp(a)-1$. Noting
  that for all $t$, we have that $g_t(0)=0$ and 
  $$\frac{\partial g_t}{\partial a}(0) = - \lambda_1(t) - 2\lambda_2(t) + \mu(t) < - \varepsilon$$
  by \eqref{eq:812}, we find $a_c>0$ such that $g_t(a)<0$ for all
  $0<a<a_c$ and all $t\geq 0$, which means that
  $(\exp(-a C_t))_{t\geq 0}$ is an
  $(\mathcal F_t)_{t\geq 0}$-super-martingale. By the martingale
  convergence theorem, it converges almost surely and -- since the sum
  of rates is bounded away from~$0$ -- the only possible almost sure
  limit is~0. Now, 1.\ follows since
  $C_t\xrightarrow{t\to\infty} \infty$ if and only if
  $\exp(-aC_t)\xrightarrow{t\to\infty} 0$ for some $a>0$.  For 2., the
  process $(\exp(-aC_{t\wedge T_1}))_{t\geq 0}$ is a non-negative
  supermartingale by Optional Stopping. Let us assume that
  $T_1<\infty$ almost surely, which occurs if and only if
  $C_{t\wedge T_1}\xrightarrow{t\to\infty} 1$ almost surely. Then,
  using the Optional Stopping Theorem, we obtain with $C_0=2$
  \begin{align*}
    \exp(-2a)=E[\exp(-a C_0)]&\geq \lim_{t\rightarrow\infty}
    E[\exp(-a C_{t\wedge T_1})]\\
    &= E[\lim_{t\rightarrow\infty}
    \exp(-a C_{t\wedge T_1})]=\exp(-a),
  \end{align*}
  a contradiction since $a>0$. Thus, we have that $P(T_1=\infty)>0$.
\end{proof}

\section{Proofs}
Here, we will show our main results.

\subsection{Proof of Theorem~\ref{thm:ext_surv}}
\label{sec:proof1}
For (i), we have $\alpha>\gamma$. The assertion is a consequence of
the coupling with the biased voter model from
Lemma \ref{l:dominance} (with bias $\gamma$ and $\alpha$). Since the
biased voter model dominates the cVMBC and type~1 dies out in the
biased voter model (\ref{rem:biasvoter}), the same is true for the
cVMBC.

The proof of (ii) is more involved. We have to show that cooperators
survive almost surely when started in a non-trivial translation
invariant configuration. Therefore, we analyze an arbitrary cluster of
cooperators and show that the size of such a cluster has a positive
probability to diverge off to infinity. Note that the flanking regions
of a cluster of cooperators can have three different forms:
\begin{align*}
  & \enskip \text{Case A \qquad \qquad \quad Case B \qquad \qquad \qquad \quad \enskip Case C} \notag\\
  &00\!\!\!\!\!\!\!\!\!\!\!\!\!\!\!\!\!\!\underbrace{1...1}_{\text{cluster
      of
      cooperators}}\!\!\!\!\!\!\!\!\!\!\!\!\!\!\!\!\!\!00,\qquad\qquad
  10\!\!\!\!\!\!\!\!\!\!\!\!\!\!\!\!\!\!\underbrace{1...1}_{\text{cluster
      of
      cooperators}}\!\!\!\!\!\!\!\!\!\!\!\!\!\!\!\!\!\!01,\qquad\qquad
  00\!\!\!\!\!\!\!\!\!\!\!\!\!\!\!\!\!\!\!\!\!\!\!\!\!\!\!\!\!\!\!\!\!\!\!\underbrace{1...1}_{\hspace*{2cm}\text{cluster
      of cooperators}}\!\!\!\!\!\!\!\!\!\!\!\!\!\!\!\!\!\!\!\!\!\!\!\!\!\!\!\!\!\!\!\!\!\!\!01 \text{ or }
  10\!\underbrace{1...1}_{\mbox{}}\!00
\end{align*}
These are the only possible environments a cluster of cooperators can
encounter in one dimension. Note that a cluster can also only consist
of a single cooperator. The dynamics of the cluster size depends on
the environment. Precisely, by the dynamics of the process, we obtain
the following. A cluster of size $x>1$
\begin{equation}
  \label{eq:511}
  \begin{aligned}
    \text{in case A } & \text{jumps to $y=x+1$ at rate $1+\gamma$}\\
    & \text{jumps to $y=x-1$ at rate $1+\alpha$}\\
    \text{in case B } & \text{jumps to $y\geq x+2$ at rate at least $2+\gamma$}\\
    & \text{jumps to $y=x-1$ at rate $1+\alpha+\gamma$}\\
    \text{in case C } & \text{jumps to $y\geq x+2$ at rate at least $1+\tfrac{\gamma}{2}$}\\
    & \text{jumps to $y=x+1$ at rate $\tfrac{1+\gamma}{2}$}\\
    & \text{jumps to $y=x-1$ at rate $1+\alpha+\tfrac{\gamma}2$}.
  \end{aligned}
\end{equation}
Under the assumptions of Theorem~\ref{thm:ext_surv}, let $(V_t)_{t\geq
  0}$ be a stochastic process representing the cluster of cooperators which is closest to the
origin and contains at least two cooperators. (If there is no such
cluster at time 0, wait for some time $\varepsilon>0$ and pick the
cluster then.) We will show that
\begin{align}
  \label{eq:510}
  P(V_t \uparrow \mathbb Z)>0.
\end{align}
For this, we compare $|V| = (|V_t|)_{t\geq 0}$ with a jump process
$(\widetilde V_t)_{t\geq 0}$ as in Corollary~\ref{cor:2}, where the jump rates at times $t$ are given as follows:
\begin{equation*}
  \begin{aligned}
    \text{in case A } & & \lambda_1(t)&=1+\gamma,& & \lambda_2(t)=0,& &  \mu(t)=1+\alpha; \\
    \text{in case B } & & \lambda_1(t)&=0, & & \lambda_2(t) = 2+\gamma,& &  \mu(t) = 1+\alpha+\gamma; \\
    \text{in case C } & & \lambda_1(t)&= \frac{1+\gamma}{2},& & \lambda_2(t) = 1 + \frac{\gamma}{2},& &  \mu(t) = 1+\alpha+\frac{\gamma}{2}.
  \end{aligned}
\end{equation*}
Moreover, this process is stopped when reaching~1. By the
comparison in \eqref{eq:511}, we see that we can couple $|V|$ and
$\widetilde V$ such that $ \widetilde V\leq |V|$, at least until
$\widetilde V$ reaches~1. Since the jump rates of $\widetilde V$ indeed satisfy $2\lambda_2(t) + \lambda_1(t) - \mu(t) > \varepsilon >0$ for all times $t\geq 0$ we find $\widetilde
V_t\xrightarrow{t\to\infty}\infty$ with positive probability which implies that $P(|V_t| \xrightarrow{t\to\infty} \infty)>0$ holds as well.
Still, we need to make sure that the cluster does not wander to $\pm
\infty$. For this, consider both boundaries of the cluster if it has
grown to a large extent. The right boundary is again bounded from
below by a jump process of the form as in Corollary~\ref{cor:2} with
$\lambda_1(t) = \frac{1+\gamma}{2}, 0;\ \lambda_2(t) = 0, 1+\frac{\gamma}{2}$ and $\mu(t)=\frac{1+\alpha}{2}, \frac{1+\alpha+\gamma}{2}$ for the cases A and B (note that the right boundary alone of case C is already captured by the right boundaries of the cases A and B). So, again, we see from
Corollary~\ref{cor:2} that the right border of the cluster goes to
infinity with positive probability. The same holds for the left
border of the cluster which tends to $-\infty$. Therefore, we have shown \eqref{eq:510}.

Now we use \eqref{eq:510} to show that defectors indeed go extinct.
Note that, from the argument given above, the probability of survival
of a cluster of cooperators depends on the environment, but is bounded
away from~0 by some $p>0$. We start at time $0$ with a cluster of
cooperators which has at least probability $p$ to survive as proved
above. In case it survives we are done, otherwise it goes extinct in
finite time and has at most merged with finitely many other
cooperating clusters until then. Thus, at this extinction time we can
choose another cluster of cooperators which exists due to the
translation invariance of the starting configuration. This cluster
again has a probability of survival of at least $p$ independently of
the history of the interacting particle system. This allows for a
Borel-Cantelli-argument showing that when repeating these steps
arbitrarily often eventually one of the cooperating clusters
survives. This happens at the latest after a geometrically distributed
number of attempts and thus in finite time. Hence, we have
$P(\lim_{t\rightarrow\infty} X_t(u)=1)=1$ for all $u$ and we are done.

For (iii), in order to prove clustering in the case
$\alpha=\gamma > 0$, there are actually two proofs. One relies on the
dual lattice and the study of process of cluster interfaces, which
performs annihilating random walks. This technique would even show
clustering for all parameters $\alpha$ and $\gamma$. However, we show
clustering by studying the probability of finding a cluster edge in
the special case $\alpha=\gamma$. Clustering for the other parameter
configurations was already shown in (i) and (ii) since extinction also
implies clustering of the process.

For our method, we write
$p_t(i_0 \cdots i_k) := P(X_t(0)=i_0,\cdots X_t(k)=i_k)$ for
$i_0,...,i_k\in\{0,1\}$ and $k=0,1,2,...$. We have to show that
\begin{align}
  \label{eq:toshow0}
  p_t(10) \xrightarrow{t\to\infty} 0, \qquad
  p_t(01)\xrightarrow{t\to\infty} 0
\end{align}
since then -- by translation invariance -- every configuration
carrying both types has vanishing probability for $t\to\infty$.

We start with the dynamics of $p_t(1)$, which reads (recall that
$\alpha=\gamma$)
\begin{align*}
  \frac{\partial p_t(1)}{\partial t} &= \frac{1}{2}(p_t(10)+p_t(01))
  +\frac{\gamma}{2}(p_t(110)+p_t(011))\\
  &\qquad \qquad \qquad \qquad \qquad-
  \frac{1+\alpha}{2}(p_t(10) + p_t(01)) - \gamma p_t(101)\\
  &= -\frac{\alpha}{2}(p_t(10)+p_t(01)) +\frac{\gamma}{2}(p_t(10) + p_t(01) - 2p_t(010)) -\gamma p_t(101)\\
  &= -\gamma (p_t(101) + p_t(010)) \leq 0.
\end{align*}
Since $p_t(1)\in [0,1]$, this probability has to converge for $t\to\infty$, hence $\frac{\partial p_t(1)}{\partial t}
\xrightarrow{t\to\infty} 0$, and therefore
\begin{equation}\label{eq:singletons}
\begin{aligned}
  &p_t(101)\xrightarrow{t\to\infty} 0, \qquad
  p_t(010)\xrightarrow{t\to\infty} 0.
\end{aligned}
\end{equation}
Now, consider the dynamics of $p_t(11)$, which is
\begin{align*}
  \frac{\partial p_t(11)}{\partial t} &= p_t(101)
  +\frac{\gamma}{2}(p_t(1101)+p_t(1011)) \\ & \qquad \qquad \qquad
  \qquad -
  \frac{1+\alpha}{2}(p_t(110) + p_t(011)) - \frac \gamma 2 (p_t(1011) + p_t(1101))\\
  &= p_t(101) - \frac{1+\alpha}{2}(p_t(110)+p_t(011)).
\end{align*}
Since we know that $p_t(101) \xrightarrow{t\to\infty} 0$
by~\eqref{eq:singletons}, and because $p_t(11)\in [0,1]$, we also have
that
\begin{align*}
  &p_t(110)\xrightarrow{t\to\infty} 0, \qquad
  p_t(011)\xrightarrow{t\to\infty} 0.
\end{align*}
We now conclude with
\begin{align*}
  p_t(10) & = p_t(010) + p_t(110) \xrightarrow{t\to\infty} 0, \\ 
  p_t(01) & = p_t(010) + p_t(011) \xrightarrow{t\to\infty} 0,
\end{align*}
which shows~\eqref{eq:toshow0}. 
\qed

\subsection{Proof of Theorem~\ref{thm:ext_surv2}}
\label{sec:proof2}
(i) We use the comparison with the biased voter model from
Lemma~\ref{l:44}. Therefore, we have that $\alpha>\gamma(d-1)/d$ if
and only if $\alpha + \gamma/(2d) > \gamma(2d-1)/(2d)$. Since for this
choice of parameters type~1 goes extinct in the biased voter model
which dominates the aVMBC, we are done.
\\
(ii) For $d=1$ and the nearest neighbor random walk, the altruistic
mechanism is such that a configuration $01$ (or $10$) turns into $00$
at rate $\alpha/2 + \gamma/4$. The same holds for the cVMBC with selection rate $\alpha + \gamma/2$. In addition, $110$ (or $011$) turns to $111$ at rate
$\gamma/2$, which is the same as for the cVMBC with cooperation
parameter $\gamma$. This shows that the transition rates for the altruistic process $\widetilde{X}$ satisfy: \\
If $\widetilde{X}(u)=0$, then
\begin{align*} 
  c(u,\widetilde X)& =\frac{1}{2}\sum_{v:|v-u|=1} \widetilde X(v)+\frac{\gamma}{4} \sum_{v:|v-u|=1}
  \widetilde X(v) \sum_{\substack{w:|w-v|=1\\ w\neq u}} \widetilde X(w).  \intertext{If
    $\widetilde X(u)=1$, then} c(u,\widetilde X) & = \frac{1+\alpha+\gamma/2}{2}\sum_{v:|v-u|=1}
  (1-\widetilde X(v))+\frac{\gamma}{4} \sum_{v:|v-u|=1}
  (1-\widetilde X(v)) \sum_{\substack{w:|w-v|=1\\ w\neq u}}
  \widetilde X(w).
\end{align*}
These resemble the transition rates of a cVMBC with selection rate $\alpha+\frac{\gamma}{2}$ and cooperation rate $\frac{\gamma}{2}$, see also equations \eqref{eq:104} and \eqref{eq:105}. In particular, clustering in the case $\gamma>\alpha=0$ 
follows from Theorem~\ref{thm:ext_surv} (iii).
\qed

\subsection{Proof of Theorem \ref{thm:complete_conv}}
\label{sec:proof3}
At time $t$, let $N_t$ be the number of finite clusters in $X_t$ with
sizes $C_t^1,...,C_t^{N_t}$. If the process starts with finitely many
defectors (cooperators), $C_t^1, C_t^3,C_t^5,...$ are sizes of
clusters of defectors (cooperators), and $C_t^2, C_t^4,...$ are sizes
of clusters of cooperators (defectors). Note that $(N_t,
C_t^1,...,C_t^{N_t})_{t\geq 0}$ is a Markov process. We will show the
following:
\begin{enumerate}
\item Either, $N_t\xrightarrow{t\to\infty}0$ or
  $N_t\xrightarrow{t\to\infty}1$. 
\item In cases (ii) and (iii), $N_t\xrightarrow{t\to\infty}0$.
\item If $N_t\xrightarrow{t\to\infty}1$, then $C_t^1
  \xrightarrow{t\to\infty}\infty$.
\end{enumerate}
Note that 1.\ and 3.\ together imply (i), i.e.\ $X$ clusters in all
cases. Of course, 2.\ implies (ii) and (iii).

1. The process $N=(N_t)_{t\geq 0}$ is non-increasing and
bounded from below by~0, so convergence of $N$ is certain. We assume
that $N_0 = n\geq 3$. Note that $N_0$ is odd and remains so until it hits $1$ from where it may or may not jump to $0$. In order to prove the claim we show that the hitting time $\inf\{s:
N_s<n\}$ is finite almost surely. For this, it suffices to show that
\begin{align}
  \label{eq:Tfinite}
  T := \inf\{s: C_s^k=1 \text{ for some $1\leq k\leq N_s$}\}<\infty
\end{align}
almost surely, since by time $T$, some cluster has size~1 and there is
a positive chance that $N$ decreases at the next transition. If $N$
does not decrease, there is the next chance after another finite time
and eventually, $N$ will decrease.

If $\alpha\geq \gamma$, consider the size $C_t$ of a cluster of
cooperators. Before time $T$, all clusters have size at least~2, so
$C_t$ jumps
\begin{align*}
  &\text{ from $C$ to $C+1$ at rate $1 + \gamma$},\\
  &\text{ from $C$ to $C-1$ at rate $1 + \alpha$},
\end{align*}
hence $(C_{t\wedge T})_{t\geq 0}$ is dominated by a symmetric random
walk with jump rate $1+\alpha$, stopped when hitting~$1$, which
implies that $T<\infty$ almost surely due to the recurrence of the
symmetric random walk in one dimension. If $\gamma\geq\alpha$, the
same argument shows that $T<\infty$ if the role of cooperators and
defectors is exchanged. Hence we have proved \eqref{eq:Tfinite} and
1.\ is shown.

2. If $N_t\xrightarrow{t\to\infty}1$ and $\alpha\geq \gamma$, the
remaining finite cluster must consist of defectors (since the argument
used in 1. shows that a finite cluster of cooperators would die out).
Therefore, in case (ii), we must have that $N_t\xrightarrow{t\to\infty}0.$ If
$\gamma\geq \alpha$, the remaining finite cluster consists of cooperators
for the same reason. Hence, in (iii), we must have that
$N_t\xrightarrow{t\to\infty}0.$ Thus, we have shown 2. 

3. As just argued in 1.\ and 2.\, if $N_t\xrightarrow{t\to\infty}1$ the
remaining finite cluster must contain the stronger type, i.e.\
defectors for $\alpha > \gamma$ and cooperators for
$\gamma>\alpha$. The size of the remaining finite cluster therefore is
a biased random walk which goes to infinity on
$\{N_t \xrightarrow{t\to\infty}1\}$ and the result follows.  \qed

\subsubsection*{Acknowledgments}
We would like to thank two anonymous referees for thorough reading and useful comments which improved the manuscript.
This research was supported by the DFG through grant Pf-672/5-1.

\bibliographystyle{alpha}

\begin{thebibliography}{ELSS02}


\bibitem[AS12]{archetti}
M. Archetti and I. Scheuring.
\newblock {\em Review: Game theory of public goods in one-shot social dilemmas without assortment}.
\newblock Journal of Theoretical Biology, Vol. 299 (2012), 9--20.

\bibitem[BK11]{blath}
J. Blath and N. Kurt.
\newblock {\em Survival and extinction of caring double-branching annihilating random walk}.
\newblock Electron. Commun. Probab. 16 (2011), no. 26, 271--282.

\bibitem[BG81]{bramson}
M. Bramson and D. Griffeath.
\newblock {\em On the Williams-Bjerknes Tumour Growth Model I}.
\newblock Ann. Probab. 9 (1981), no. 2, 173--185.

\bibitem[BBRG08]{brockhurst}
M.~A. Brockhurst, A. Buckling, D. Racey and A. Gardner.
\newblock {\em Resource supply and the evolution of public-goods cooperation in bacteria}.
\newblock BMC Biology 6:20 (2008).

\bibitem[Clu09]{clutton}
T. Clutton-Brock.
\newblock {\em Cooperation between non-kin in animal societies}.
\newblock Nature 462 (2009), no. 7269, 51--57.

\bibitem[Cre01]{crespi}
B.~J. Crespi.
\newblock {\em The evolution of social behavior in microorganisms}.
\newblock Trends in Ecology \& Evolution, Vol. 16 (2001), no.4, 178--183.

\bibitem[Czu16]{dissertation}
P. Czuppon.
\newblock {\em Phenotypic heterogeneity in bacterial populations -- a mathematical study}.
\newblock University of Freiburg, Dissertation (2016), DOI: 10.6094/UNIFR/11207.

\bibitem[DNSWB14]{drescher}
K. Drescher, C.~D. Nadell, H.~A. Stone, N.~S. Wingreen and B.~L. Bassler.
\newblock {\em Solutions to the Public Good Dilemma in Bacterial Biofilms}.
\newblock Current Biology 24 (2014), no. 1, 50--55. 

\bibitem[EK86]{ethier}
S.~N. Ethier and T.~G. Kurtz.
\newblock {\em Markov processes. Characterization and convergence}.
\newblock Wiley Series in Probability and Mathematical Statistics: Probability
  and Mathematical Statistics. John Wiley \& Sons Inc., New York, 1986.

\bibitem[EL16]{lanchier}
S. Evilsizor and N. Lanchier.
\newblock {\em Evolutionary games on the lattice: death-birth updating process}.
\newblock Electron. J. Probab. 21 (2016), paper no. 17, 1--29.

\bibitem[GW03]{griffin}
A.~S. Griffin and S.~A. West.
\newblock {\em Kin Discrimination and the Benefit of Helping in Cooperatively Breeding Vertebrates}.
\newblock Science 302 (2003), no. 5645, 634--636.

\bibitem[HJM15+]{hutzenthaler}
M. Hutzenthaler, F. Jordan, D. Metzler.
\newblock{\em Altruistic defense traits in structured populations}.
\newblock arXiv:1505.02154, Mathematics - Probability, 2015.

\bibitem[Lig85]{liggett}
T.~M. Liggett.
\newblock {\em Interacting Particle Systems}.
\newblock Springer Berlin Heidelberg, 1985.

\bibitem[LTV14]{louidor}
O. Louidor, R. Tessler, A. Vandenberg-Rodes.
\newblock {\em The Williams–Bjerknes model on regular trees}.
\newblock Ann. Appl. Probab. 24 (2014), no. 5, 1889--1917.

\bibitem[Now06]{nowak}
M. Nowak. 
\newblock {\em Five Rules for the Evolution of Cooperation}.
\newblock Science 314 (2006), no. 5805, 1560--1563. 

\bibitem[PF10]{penn}
D.~J. Penn and J.~G. Frommen. 
\newblock {\em Kin recognition: an overview of conceptual issues, mechanisms and evolutionary theory}.
\newblock Animal Behaviour: Evolution and Mechanisms, Chapter 3. Springer Berlin Heidelberg, 2010.

\bibitem[SS15]{sturm}
A. Sturm and J.~M. Swart.
\newblock {\em A particle system with cooperative branching and coalescence}. 
\newblock Ann. Appl. Probab. 25 (2015), no. 3, 1616--1649. 

\bibitem[WB71]{williams}
T. Williams and R. Bjerknes.
\newblock {\em Stochastic model for abnormal clone spread through epithelial basal layer}.
\newblock Nature 236 (1972), no. 5340, 19--21.

\bibitem[WL06]{levin}
N.~S. Wingreen and S.~A. Levin.
\newblock {\em Cooperation among Microorganisms}.
\newblock PLoS Biology 4 (2006), no. 9, e299.

\end{thebibliography}

\end{document}